\documentclass[12pt,reqno]{article}

\usepackage[usenames]{color}
\usepackage{amssymb}
\usepackage{amsmath}
\usepackage{amsthm}
\usepackage{amsfonts}
\usepackage{amscd}
\usepackage{graphicx}
\usepackage[colorlinks=true,
linkcolor=webgreen,
filecolor=webbrown,
citecolor=webgreen]{hyperref}

\definecolor{webgreen}{rgb}{0,.5,0}
\definecolor{webbrown}{rgb}{.6,0,0}

\usepackage{color}
\usepackage{fullpage}
\usepackage{float}

\usepackage{graphics}
\usepackage{latexsym}

\begin{document}

\theoremstyle{plain}
\newtheorem{theorem}{Theorem}
\newtheorem{corollary}[theorem]{Corollary}
\newtheorem{proposition}{Proposition}
\newtheorem{lemma}{Lemma}
\newtheorem{example}{Example}
\newtheorem{remark}{Remark}

\begin{center}
{\bf \large
Abel-Type Transformations and Telescoping Structures\\[6pt] in Reciprocal Series of Second-Order Linear Recurrences}

\vskip 1cm

{\large
Kunle Adegoke \\
Department of Physics and Engineering Physics \\ Obafemi Awolowo University, 220005 Ile-Ife, Nigeria \\
\href{mailto:adegoke00@gmail.com}{\tt adegoke00@gmail.com}

\vskip 0.06 in

Robert Frontczak \\
Independent Researcher, 72764 Reutlingen, Germany \\
\href{mailto:robert.frontczak@web.de}{\tt robert.frontczak@web.de}

\vskip 0.06 in

Taras Goy \\
Faculty of Mathematics and Computer Science\\
Vasyl Stefanyk Carpathian National University \\
76018 Ivano-Frankivsk, Ukraine \\[3pt]
\href{mailto:taras.goy@pnu.edu.ua}{\tt taras.goy@pnu.edu.ua}
}

\end{center}

\vskip .0 in

\begin{abstract}
We develop a unified method for transforming and evaluating infinite series involving products of terms of second-order linear recurrences in the denominator. The approach is based on a discrete Abel-type summation formula (summation by parts), which converts reciprocal series with three or more factors into expressions exhibiting a partial telescoping structure. As a consequence, we obtain general transformation formulas for series of the form
$\sum\limits_{k=1}^{\infty} \frac{(\pm 1)^k}{w_{rk+l}\, w_{mk+s}\, w_{m(k+1)+s}}$,
together with extensions to products of four or more terms. These formulas provide a systematic framework that unifies and extends many known identities for Fibonacci and Lucas numbers.

In addition, the method leads to explicit evaluations and identities involving several classical combinatorial sequences, including Catalan numbers, harmonic numbers, and Stirling numbers of both kinds.

A key feature of the approach is that it naturally distinguishes between even and odd values of the parameter $m$, leading to structurally different representations. The results show that summation by parts is an effective and flexible tool for reducing multi-factor reciprocal sums to simpler forms.
\end{abstract}


\section{Introduction}

Infinite series involving second-order linear recurrence sequences such as the Fibonacci, Lucas, and more general Horadam sequences have been studied extensively in number theory, combinatorics, and the theory of special functions. These series arise naturally in 
connection with generating functions, continued fractions, and various summation identities, and they often admit surprisingly explicit closed forms.

A particularly rich class of such identities consists of reciprocal sums involving products of consecutive or shifted terms, for example,
	\begin{equation*} 
	\sum_{k=1}^{\infty}\frac{1}{F_k F_{k+1}}, \qquad \sum_{k=1}^{\infty}\frac{1}{F_k F_{k+1} F_{k+2}},
	\end{equation*}
is called nowadays ``Brousseau sums''. These sums are named after Brousseau, who was the first to study them in the 
late 1960s \cite{Brousseau1967,Brousseau1967b}. More generally, Brousseau investigated 
the series
	\begin{equation*}
	\sum_{n=1}^{\infty}\frac{(\pm1)^n}{F_nF_{n+k_1}F_{n+k_2}\cdots F_{n+k_m}},
	\end{equation*}
where $(F_n)$ are Fibonacci numbers, $k_j$ are positive integers with $k_1<k_2<\ldots<k_m$. He, among other results, proved that
	\begin{gather*}
	\sum_{n=1}^{\infty}
	\frac{(-1)^n}{F_{n} F_{n+s}}
	=\frac{1}{F_s}
	\left(\frac{1-\sqrt5}{2}\,s+\sum_{j=1}^{s}\frac{F_{j-1}}{F_j}
	\right),\qquad s\geq1,\\
	\sum_{n=1}^{\infty}
	\frac{1}{F_{n}F_{n+1}F_{n+2}F_{n+3}F_{n+4}}
	=\frac{5}{18}+\frac23
	\sum_{n=1}^{\infty
	}\frac{(-1)^n}{F_{n}F_{n+1}F_{n+2}},
	\end{gather*}
and
	\begin{equation*}
	\sum_{n=1}^{\infty}
	\frac{1}{F_{n} F_{n+1}F_{n+2}F_{n+3}F_{n+4}F_{n+5}F_{n+6}F_{n+7}F_{n+8}}
	=
	\frac{319}{16380}
	\sum_{n=1}^{\infty}\frac{1}{F_n}
	-\frac{46816051}{715478400}.
	\end{equation*}

Shortly after the work of Brousseau, Carlitz \cite{Carlitz-71} derived, for every nonnegative integer $s$, representations of the series
	\begin{equation*}
	\sum_{n=1}^{\infty}
	\frac{(\pm1)^n}{F_nF_{n+1}\cdots F_{n+4k+s}},
	\qquad k\ge1,\quad s\in\{0,1,2,3\},
	\end{equation*}
in terms of Fibonomial coefficients.
Melham in \cite {Melham2002}
investigated the series
	\begin{equation*}
	\sum_{n=1}^{\infty}\frac{(\pm1)^n}{W_nW_{n+k_1}W_{n+k_2}\cdots W_{n+k_m}},
	\end{equation*}
where $(W_n)$ is a Horadam sequence (general second-order linear recurrence) and the integers $k_1<k_2<\ldots<k_m$ are positive.

In \cite{AFG1}, the authors of this paper evaluated several three-parameter families of reciprocal Horadam sums in closed form. Adegoke \cite{Adegoke} obtained several closed-form evaluations for infinite Fibonacci--Lucas series involving products of shifted Fibonacci or Lucas numbers in the denominator. In particular, for positive integers $m$, $n$, and $q$, formulas were obtained for series such as
	\begin{gather*}
	\sum_{k=1}^{\infty}
	\frac{(\pm1)^k}{F_{nk} F_{nk+nq} \cdots F_{nk+(m-1)nq} F_{nk+(m+1)nq} \cdots F_{nk+2mnq}},\\
	\sum_{k=1}^{\infty}
	\frac{(\pm1)^k}{L_{nk} L_{nk+nq} \cdots L_{nk+(m-1)nq} L_{nk+(m+1)nq} \cdots L_{nk+2mnq}},\\
	\sum_{k=1}^{\infty}
	\frac{(\pm1)^k L_{nk+mnq}}{F_{nk} F_{nk+nq} \cdots F_{nk+2mnq}},\qquad \sum_{k=1}^{\infty}
	\frac{(\pm1)^k F_{nk+mnq}}{L_{nk} L_{nk+nq} \cdots L_{nk+2mnq}}, 
	\end{gather*}
where $(L_n)$ denotes the sequence of Lucas numbers.

Numerous other results for this type of sums have been obtained in the literature, including both finite and infinite sums, alternating variants, and identities involving subsequences. The literature is vast and we can only present an incomplete selection of chronologically sorted articles: Good \cite{Good}, Bruckman and Good \cite{Bruckman}, Hoggatt and Bicknell \cite{Hoggatt1,Hoggatt2}, Melham and Shannon \cite{Melham-Shannon}, Andr\'{e}-Jeannin \cite{A-Jeannin}, Rabinowitz \cite{Rabinowitz1,Rabinowitz2}, Zhao \cite{Zhao}, Hu, Sun and Liu \cite{Hu}, Melham \cite{Melham2000,Melham2001,Melham2002,Melham2012,Melham2013,Melham2014a,Melham2014b,Melham2015a,Melham2015b}, K\i l\i \c c and Prodinger \cite{Kilic1}, Frontczak \cite{Frontczak2,Frontczak1}, Adegoke \cite{Adegoke}, Farhi \cite{Farhi}, K\i l\i \c c and Ersanl\i{} \cite{Kilic2}, Adegoke, Frontczak and Goy \cite{AFG-Tatra,AFG1},
Adegoke, Frontczak and Gryszka \cite{AFG2}.
It is also important to mention that additional appealing results have appeared in journals as problem proposals. Here we mention Ohtsuka with three recent results \cite{Ohtsuka1,Ohtsuka2,Ohtsuka3} who published many more and must be seen as an almost endless source for these relations.

However, most results for reciprocal sums with several factors in the denominator are obtained by ad hoc methods, often tailored to specific sequences or particular choices of parameters. As a result, a general framework that systematically explains and produces such identities is still lacking.

This paper aims to provide a unified approach to this problem based on a discrete Abel-type summation formula (summation by parts). 
The central idea is to transform reciprocal series involving three or more terms in the denominator into alternating sums that exhibit a 
partial telescoping structure. This transformation reduces the complexity of the original expressions and allows one to derive families of identities in a uniform and transparent way. Our main results provide general transformation formulas for infinite 
series of the form
	\begin{equation*}
	\sum_{k=1}^{\infty} \frac{(\pm1)^k}{w_{rk+l}  w_{mk+s}  w_{m(k+1)+s}},
	\end{equation*}
where $w_n$ is a second-order recurrent sequence, as well as extensions to series involving four or more factors in the denominator. A key feature of the method is that it naturally distinguishes between even and odd values of the parameter $m$, leading to qualitatively different representations in these two cases. These transformation formulas yield, as special cases, a variety of explicit identities for Fibonacci and Lucas numbers, including both classical formulas and new relations. In contrast to previous approaches, which typically rely on sequence-specific arguments, our method applies uniformly to a broad class of second-order recurrence sequences and systematically produces alternating and partially telescoping representations.

The paper is organized as follows. In Section 2 we establish the main summation lemma and derive transformation formulas for series with three factors. Section 3 is devoted to extensions to four or more factors and to several explicit examples and applications. In Section~4 we present additional applications involving other prominent sequences such as Catalan numbers and Stirling numbers. In Section 5 we investigate further applications of the Abel-type transformation by establishing a symmetric identity and deriving several consequences for general second-order recurrence sequences and, in particular, for the Fibonacci and Lucas numbers. Section 6 concludes our study.

\section{Transformation Formulas for Reciprocal Series}\label{Sect2}

Let $(w_n)=(w_n(a,b;p))$
be the second-order recurrence sequence defined by
\begin{equation*}
	w_n=pw_{n-1}+w_{n-2}, \qquad n\ge2,
	\end{equation*}
with initial values $w_0=a$ and $w_1=b$.

Let $(u_n)$ and $(v_n)$ denote the associated companion sequences defined by
	\begin{equation*}
	u_0=0,\quad u_1=1,\quad
	u_n=pu_{n-1}+u_{n-2},
	\end{equation*}
and
\begin{equation*}
	v_0=2,\quad v_1=p,\quad
	v_n=pv_{n-1}+v_{n-2},
	\end{equation*}
for $n\ge2$. Important special cases are
$F_n=w_n(0,1;1)$, $L_n=w_n(2,1;1)$, $P_n=w_n(0,1;2)$, $Q_n=w_n(2,2;2)$,
which are respectively the Fibonacci, Lucas, Pell, and Pell--Lucas sequences.
Certain polynomial sequences also arise as special cases of $(w_n)$. In particular,
$F_n(x)=w_n(0,1;x)$, $L_n(x)=w_n(2,x;x)$, $P_n(x)=w_n(0,1;2x)$, $Q_n(x)=w_n(2,2x;2x)$, which are respectively the Fibonacci, Lucas, Pell, and Pell--Lucas polynomials.

The main technical ingredient of the paper is a discrete Abel-type transformation (summation by parts), which allows one to convert reciprocal sums into forms exhibiting partial telescoping or alternating cancellation. This approach provides a unified framework for treating series involving products of recurrence terms in the denominator.
\begin{lemma}\label{main_lem}
	Let $(x_n)$ and $(y_n)$ be sequences of complex numbers. Then
	\begin{equation}\label{Lemma1_1}
		\sum_{k=1}^n x_k y_k = x_{n+1} \sum_{j=1}^n y_j - \sum_{k=1}^n (x_{k+1} - x_k)\sum_{j=1}^k y_j
		\end{equation}
and
	\begin{equation}\label{Lemma1_2}
		\sum_{k=1}^n x_k y_k = x_{n} \sum_{j=1}^n (-1)^{n-j} y_j + \sum_{k=1}^{n-1} (x_{k+1} + x_k)\sum_{j=1}^k (-1)^{k-j} y_j.
		\end{equation}
		
	In particular, if $\lim_{n\to\infty} {x_{n + 1} \sum_{j = 1}^n {y_j }}=0$, then
	
		\begin{equation}\label{Lem1_infty}
		\sum_{k = 1}^\infty  {x_k y_k }  = \sum_{k = 1}^\infty  {\left( {x_k  - x_{k + 1} } \right)\sum_{j = 1}^k {y_j } } .
		\end{equation}
\end{lemma}
\begin{proof} We apply the discrete summation by parts formula
		\begin{equation*}
		\sum_{k=1}^n f_k (g_{k+1}-g_k) = f_{n+1}g_{n+1}-f_1 g_1 - \sum_{k=1}^n g_{k+1}(f_{k+1}-f_k),
		\end{equation*}
	with $f_k = x_k$ and $g_k = \sum_{j=1}^{k-1} y_j$. 
	Since $g_{k+1}-g_k = y_k$, this yields \eqref{Lemma1_1}.
	
	The identity \eqref{Lemma1_2} follows from \eqref{Lemma1_1} by applying the transformations $x_k \mapsto (-1)^k x_k$ and $y_k \mapsto (-1)^k y_k$, followed by a straightforward rearrangement.
\end{proof}

Although elementary, Lemma \ref{main_lem} provides a flexible tool for transforming reciprocal sums. In particular, it allows one to systematically reduce expressions involving three or more factors in the denominator to sums with an inherent telescoping or alternating structure.

We now apply this transformation to series involving Horadam-type sequences.
\begin{theorem}\label{thm:main_transform}
	Assume the parameters are chosen so that all involved series converge. Let $r$, $s$, and $l$ be nonnegative integers.
	
	If $m$ is odd, then
		\begin{equation}\label{Th1_1}
		\sum_{k=1}^\infty
		\frac{1}{w_{rk+l}w_{mk+s}w_{m(k+1)+s}}
		=
		\frac{1}{u_m w_{m+s}}
		\sum_{k=1}^\infty (-1)^{k-1}
		\frac{\bigl(w_{r(k+1)+l}+w_{rk+l}\bigr)u_{mk}}
		{w_{rk+l} w_{r(k+1)+l} w_{m(k+1)+s}};
		\end{equation}
	while, if $m$ is even, then
		\begin{equation}\label{Th1_2}
		\sum_{k=1}^\infty
		\frac{(-1)^k}{w_{rk+l}w_{mk+s}w_{m(k+1)+s}}
		=
		\frac{1}{u_m w_{m+s}}
		\sum_{k=1}^\infty (-1)^k
		\frac{\bigl(w_{r(k+1)+l}+w_{rk+l}\bigr)u_{mk}}
		{w_{rk+l} w_{r(k+1)+l} w_{m(k+1)+s}}.
		\end{equation}
\end{theorem}
\begin{proof}
	To prove \eqref{Th1_1}, we apply the first identity of Lemma \ref{main_lem} with
	$x_k=\frac{(-1)^k}{w_{rk+l}}$
	{and}
	$y_k=\frac{(-1)^k}{w_{mk+s}w_{m(k+1)+s}}$.
	Then
		\begin{align}
		\sum_{k=1}^n \frac{1}{w_{rk+l}w_{mk+s}w_{m(k+1)+s}}
		&=
		\frac{(-1)^n}{w_{r(n+1)+l}}
		\sum_{j=1}^n
		\frac{(-1)^{j-1}}{w_{mj+s}w_{m(j+1)+s}}
		\nonumber\\
		&\quad+
		\sum_{k=1}^{n}(-1)^k
		\left(
		\frac{1}{w_{r(k+1)+l}}
		+\frac{1}{w_{rk+l}}
		\right)
		\sum_{j=1}^{k}
		\frac{(-1)^j}{w_{mj+s}w_{m(j+1)+s}}.
		\label{Th1_2_start}
		\end{align}
		
	To evaluate the inner sum, we use the identity \cite{Horadam1965}
		\begin{equation}
		w_t w_{q+1} - w_{t+1}w_q
		=
		(-1)^q E u_{t-q},
		\label{Horadam}
		\end{equation}
where $E=b^2-a^2-abp$.
	
	Taking $t=m(j+1)+s$ and $q=mj+s$, we obtain
	    \begin{equation*}
		w_{m(j+1)+s}w_{mj+s+1}
		-
		w_{m(j+1)+s+1}w_{mj+s}
		=
		(-1)^{mj+s}E u_m,
		\end{equation*}
	which yields
	\begin{equation*}
		\frac{1}{w_{mj+s}w_{m(j+1)+s}}
		=
		\frac{(-1)^{mj+s}}{E u_m}
		\left(
		\frac{w_{mj+s+1}}{w_{mj+s}}
		-
		\frac{w_{m(j+1)+s+1}}{w_{m(j+1)+s}}
		\right).
		\end{equation*}

	Multiplying by $(-1)^j$ and summing from $j=1$ to $k$, we obtain, for odd $m$, the telescoping sum
		\begin{equation*}
		\sum_{j=1}^k
		\frac{(-1)^j}{w_{mj+s}w_{m(j+1)+s}}
		=
		\frac{(-1)^s}{E u_m}
		\left(
		\frac{w_{m+s+1}}{w_{m+s}}
		-
		\frac{w_{m(k+1)+s+1}}{w_{m(k+1)+s}}
		\right),
		\end{equation*}
where the cancellation follows from the alternating structure.
	
Substituting this into \eqref{Th1_2_start} gives
		\begin{align*}
		&\sum_{k=1}^n
		\frac{1}{w_{rk+l}w_{mk+s}w_{m(k+1)+s}}
		\notag\\
		&\quad\qquad=
		\frac{(-1)^n}{w_{r(n+1)+l}}
		\sum_{j=1}^n
		\frac{(-1)^{j-1}}{w_{mj+s}w_{m(j+1)+s}}
		\notag\\
		&\qquad\qquad+
		\frac{(-1)^s}{E u_m}
		\sum_{k=1}^{n}
		(-1)^{mk}
		\frac{w_{rk+l}+w_{r(k+1)+l}}
		{w_{r(k+1)+l}w_{rk+l}}
		\left(
		\frac{w_{m+s+1}}{w_{m+s}}
		-
		\frac{w_{m(k+1)+s+1}}{w_{m(k+1)+s}}
		\right)
		\notag\\
		&\quad\qquad=
		\frac{(-1)^n}{w_{r(n+1)+l}}
		\sum_{j=1}^n
		\frac{(-1)^{j-1}}{w_{mj+s}w_{m(j+1)+s}}
		\notag\\
		&\qquad\qquad+
		\frac{(-1)^s}{E u_m}
		\sum_{k=1}^{n}
		(-1)^k
		\frac{w_{rk+l}+w_{r(k+1)+l}}
		{w_{r(k+1)+l}w_{rk+l}}
		\cdot
		\frac{
			w_{m+s+1}w_{m(k+1)+s}
			-
			w_{m(k+1)+s+1}w_{m+s}
		}
		{w_{m+s}w_{m(k+1)+s}}.
		\end{align*}
	
	Using \eqref{Horadam} again with $t=m(k+1)+s$ and $q=m+s$, we then obtain
		\begin{align*}
		\sum_{k=1}^n
		\frac{1}{w_{rk+l}w_{mk+s}w_{m(k+1)+s}}
		&=
		\frac{(-1)^n}{w_{r(n+1)+l}}
		\sum_{j=1}^n
		\frac{(-1)^{j-1}}{w_{mj+s}w_{m(j+1)+s}}
		\\
		&\quad
		-\frac{1}{u_m w_{m+s}}
		\sum_{k=1}^{n}
		(-1)^k
		\frac{(w_{rk+l}+w_{r(k+1)+l})u_{mk}}
		{w_{r(k+1)+l}w_{rk+l}w_{m(k+1)+s}}.
		\end{align*}
	
	Finally, letting $n\to\infty$, the boundary term vanishes, and thus \eqref{Th1_1} follows.
	
	The proof for even $m$ is analogous and is obtained by substituting
	$x_k=\frac{(-1)^k}{w_{rk+l}}$
	and $y_k=\frac{1}{w_{mk+s}w_{m(k+1)+s}}$ 	into \eqref{Lemma1_1}.
\end{proof}

Theorem \ref{thm:main_transform} shows that reciprocal series involving three Horadam-like numbers can be transformed into alternating sums with a built-in partial telescoping structure. This representation is particularly useful for deriving explicit identities. By selecting appropriate values for $r$, $l$, $s$, and $m$, we obtain the following classical-type and alternating identities for Fibonacci and Lucas numbers.
\begin{corollary}
		\begin{gather*}
		\sum_{k=1}^\infty \frac{1}{F_{k}^2 F_{k+1}} = \sum_{k=1}^\infty (-1)^{k-1} \frac{F_{k+2}}{F_{k+1}^2}, \qquad\qquad
		\sum_{k=1}^\infty \frac{1}{F_{k} F_{k+1} F_{k+2}} = \sum_{k=1}^\infty \frac{(-1)^{k-1}}{F_{k+1}}, \\
		\sum_{k=1}^\infty \frac{1}{L_{k}^2 L_{k+1}} = \sum_{k=1}^\infty (-1)^{k-1} \frac{F_{k}L_{k+2}}{L_k L_{k+1}^2}, \qquad\qquad
		\sum_{k=1}^\infty \frac{1}{L_{k} L_{k+1} L_{k+2}} = \frac{1}{3} \sum_{k=1}^\infty (-1)^{k-1} \frac{F_k}{L_k L_{k+1}}.
		\end{gather*}
\end{corollary}
\begin{proof}
	Substitute $m=1$ and $(r,s,l)=(1,0,0)$ and $(1,1,1)$ into \eqref{Th1_1}, separately for the Fibonacci and Lucas numbers.
\end{proof}
\begin{corollary}
		\begin{gather*}
		\sum_{k=1}^\infty \frac{(-1)^k}{F_{k}F_{2k} F_{2k+2}} = \sum_{k=1}^\infty (-1)^{k} \frac{F_{k+2}L_{k}}{F^2_{k+1}L_{k+1}}, \qquad\qquad
		\sum_{k=1}^\infty \frac{(-1)^k}{F^2_{2k} F_{2k+2}} = \sum_{k=1}^\infty (-1)^{k} \frac{L_{2k+1}}{F^2_{2k+2}}, \\
		\sum_{k=1}^\infty \frac{(-1)^k}{L_{k}L_{2k} L_{2k+2}} = \frac{1}{3} \sum_{k=1}^\infty (-1)^{k} \frac{L_{k+2}F_{k}}{L_{k+1}L_{2k+2}}, \qquad\qquad
		\sum_{k=1}^\infty \frac{(-1)^k}{L^2_{2k} L_{2k+2}} = \frac{5}{3} \sum_{k=1}^\infty (-1)^{k} \frac{F_{2k} F_{2k+1}}{L_{2k} L^2_{2k+2}}.
		\end{gather*}
\end{corollary}
\begin{proof}
	Apply \eqref{Th1_2} with $m=2$ and $(r,s,l)=(1,0,0)$ and $(2,0,0)$, considering separately the cases of the Fibonacci and Lucas numbers.
\end{proof}

A complementary transformation is obtained by applying the Abel-type identity in a different form. 
\begin{theorem}\label{thm:second_transform}
	Assume that all series involved converge. Let $r$, $s$, and $l$ are nonnegative integers.
	
	If $m$ is even, then
		\begin{equation}\label{Th2_1}
		\sum_{k=1}^\infty 
		\frac{1}{w_{rk+l}w_{mk+s}w_{m(k+1)+s}} 
		= \frac{1}{u_m w_{m+s}}\sum_{k=1}^\infty 
		\frac{\big(w_{r(k+1)+l}- w_{rk+l}\big)u_{mk}}{w_{rk+l} w_{r(k+1)+l} w_{m(k+1)+s}};
		\end{equation}
	while if $m$ is odd, then
		\begin{equation}\label{Th2_2}
		\sum_{k=1}^\infty 
		\frac{(-1)^{k-1}}{w_{rk+l}w_{mk+s}w_{m(k+1)+s}} 
		= \frac{1}{u_m w_{m+s}}\sum_{k=1}^\infty 
		\frac{\big(w_{r(k+1)+l}- w_{rk+l}\big)u_{mk}}{w_{rk+l} w_{r(k+1)+l} w_{m(k+1)+s}}.
		\end{equation}
\end{theorem}
\begin{proof}
	The proof follows the same lines as that of Theorem \ref{thm:main_transform}, with \eqref{Lemma1_2} replacing \eqref{Lemma1_1}.
	For example, in the case where $m$ is even, we obtain
	\begin{align*}
		\sum_{k=1}^n
		\frac{1}{w_{rk+l}w_{mk+s}w_{m(k+1)+s}}
		&=
		\frac{1}{w_{rn+l}}
		\sum_{j=1}^n
		\frac{1}{w_{mj+s}w_{m(j+1)+s}}
		\\
		&\quad
		-\frac{1}{u_m w_{m+s}}
		\sum_{k=1}^{n-1}
		\frac{(w_{rk+l}-w_{r(k+1)+l})u_{mk}}
		{w_{r(k+1)+l}w_{rk+l}w_{m(k+1)+s}}.
		\end{align*}
	
	Passing to the limit as $n\to\infty$, we obtain \eqref{Th2_1}.
	
	The case where $m$ is odd is treated analogously and is therefore omitted.
\end{proof}

As a further illustration of Theorem \ref{thm:second_transform}, we present several explicit identities for Fibonacci and Lucas numbers.
The following non-alternating and alternating identities are direct consequences of \eqref{Th2_1} and \eqref{Th2_2}.
\begin{corollary}
		\begin{gather*}
		\sum_{k=1}^\infty \frac{1}{F_{k} F_{2k} F_{2k+2}} = \sum_{k=1}^\infty \frac{F_{k-1}L_k}{L_{k+1} F_{k+1}^2}, \qquad\qquad 
		\sum_{k=1}^\infty \frac{1}{F_{2k}^2 F_{2k+2}} = \sum_{k=1}^\infty \frac{F_{2k+1}}{F_{2k+2}^2}, \\
		\sum_{k=1}^\infty \frac{1}{L_{k} L_{2k} L_{2k+2}} = \frac{1}{3} \sum_{k=1}^\infty \frac{F_{k}L_{k-1}}{L_{k+1} L_{2k+2}}, \qquad\qquad 
		\sum_{k=1}^\infty \frac{1}{L_{2k}^2 L_{2k+2}} = \frac{1}{3} \sum_{k=1}^\infty \frac{F_{2k}L_{2k+1}}{L_{2k} L_{2k+2}^2}.
		\end{gather*}
\end{corollary}
\begin{proof}
	Substitute $m=2$ and $(r,s,l)=(1,0,0)$ and $(2,0,0)$ into \eqref{Th2_1}, separately for the Fibonacci and Lucas numbers.
\end{proof}
\begin{corollary}
	\begin{gather*}
		\sum_{k=1}^\infty \frac{(-1)^{k-1}}{F^2_{k} F_{k+1}} = \sum_{k=1}^\infty \frac{F_{k-1}}{F^2_{k+1}}, \qquad\qquad 
		\sum_{k=1}^\infty \frac{(-1)^{k-1}}{L_{k}L_{k+1}L_{k+2}} = \frac13\sum_{k=1}^\infty \frac{L_{k-1}F_{k}}{L_{k}L_{k+1}L_{k+2}}, \\
		\sum_{k=1}^\infty \frac{(-1)^{k-1}}{F_{k} F_{k+1} F_{k+2}} = \sum_{k=1}^\infty \frac{F_{k-1}}{F_{k+1} F_{k+2}}, \qquad\qquad 
		\sum_{k=1}^\infty \frac{(-1)^{k-1}}{L_{k}^2 L_{k+1}} =  \sum_{k=1}^\infty \frac{L_{k-1}F_{k}}{L^2_{k+1} L_{k}}.
		\end{gather*}
\end{corollary}
\begin{proof}
	The result follows by applying \eqref{Th2_2} with $m=1$ and choosing $(r,s,l)=(1,0,0)$ and $(2,0,0)$, respectively, for the Fibonacci and Lucas numbers.
\end{proof}

Theorems \ref{thm:main_transform} and \ref{thm:second_transform} provide two complementary transformation schemes: one based on alternating telescoping, and another based on finite differences of the numerator sequence. These dual representations form the basis for the extensions developed in the next section.

\section{Series with Four and More Factors}

The transformation method developed in the previous section extends naturally to reciprocal series involving four or more factors. In this case, the resulting identities exhibit a higher-order telescoping structure, reflecting the interaction between consecutive terms of the underlying recurrence.

We first consider the case of four factors.
\begin{theorem}\label{thm:four_factors}
	Assume that all series involved converge. Let $r$, $s$, and $l$ are nonnegative integers. 
	
	If $m$ is odd, then	
		\begin{align}\label{Th3_1}
		&\sum_{k=1}^\infty 
		\frac{1}{w_{rk+l}w_{r(k+1)+l}w_{mk+s}w_{m(k+1)+s}}\notag\\
		&\qquad\qquad= \frac{1}{u_m w_{m+s}}\sum_{k=1}^\infty (-1)^{k-1} \frac{\big(w_{rk+l}+w_{r(k+2)+l}\big)u_{mk}}
		{w_{rk+l} w_{r(k+1)+l} w_{r(k+2)+l} w_{m(k+1)+s}};
		\end{align}	
	while if $m$ is even, then	
		\begin{align}\label{Th3_2}
		&\sum_{k=1}^\infty 
		\frac{(-1)^k}{w_{rk+l}w_{r(k+1)+l}w_{mk+s}w_{m(k+1)+s}}\notag\\
		&\qquad\qquad = \frac{1}{u_m w_{m+s}}\sum_{k=1}^\infty (-1)^k \frac{\big(w_{rk+l}+w_{r(k+2)+l}\big)u_{mk}}
		{w_{rk+l} w_{r(k+1)+l} w_{r(k+2)+l} w_{m(k+1)+s}}.
		\end{align}	
\end{theorem}
\begin{proof}
	For proving \eqref{Th3_1}, we 	apply \eqref{Lemma1_1} 
	with
	$x_k=\frac{(-1)^k}{w_{rk+l}w_{r(k+1)+l}}$ and $y_k=\frac{(-1)^k}{w_{mk+s}w_{m(k+1)+s}}$.  This gives 
		\begin{align*}
		\sum_{k=1}^n &\frac{1}{w_{rk+l}w_{r(k+1)+l}w_{mk+s}w_{m(k+1)+s}}\\
		&\qquad=
		\frac{(-1)^{n}}{w_{r(n+1)+l}w_{r(n+2)+l}} \sum_{j=1}^n \frac{(-1)^{j-1}}{w_{mj+s}w_{m(j+1)+s}}\\
		&\qquad\quad +
		\sum_{k=1}^{n} (-1)^{k} \left(\frac{1}{w_{r(k+1)+l}w_{r(k+2)+l}}+\frac{1}{w_{rk+l}w_{r(k+1)+l}}\right) \sum_{j=1}^k  \frac{(-1)^j}{w_{mj+s}w_{m(j+1)+s}}.
		\end{align*}
	
	The inner sum is evaluated exactly as in the proof of Theorem \ref{thm:main_transform}. Substituting its telescoping form into the above identity, using identity \eqref{Horadam}, and letting $n\to\infty$, we obtain \eqref{Th3_1}.
	
	The proof of \eqref{Th3_2} is completely analogous and is therefore omitted.
\end{proof}

Theorem \ref{thm:four_factors} shows that the Abel-type transformation extends to higher-order products, producing identities with an additional layer of telescoping involving three consecutive terms of the sequence $(w_n)$. We now present several explicit identities for Fibonacci and Lucas numbers that follow from Theorem \ref{thm:four_factors}.
\begin{corollary}
		\begin{gather*}
		\sum_{k=1}^\infty \frac{1}{F_{k}F_{k+1}F_{k+2}F_{k+3}} = \frac{1}{2}\sum_{k=1}^\infty (-1)^{k-1}  \frac{L_{k+1}}{F_{k+1}F_{k+2}F_{k+3}},\\
		\sum_{k=1}^\infty \frac{1}{L_{k} L_{k+1} L_{k+2}L_{k+3}} = \frac{5}{4} \sum_{k=1}^\infty (-1)^{k-1} \frac{ F_{k+1}}{L_kL_{k+1}L_{k+2}L_{k+3}},\\
		\sum_{k=1}^\infty \frac{1}{F_{k}^2 F_{k+1}^2} = \sum_{k=1}^\infty (-1)^{k-1} \frac{L_{k+1}}{F_{k+1}^2F_{k+2}},\qquad\qquad 
		\sum_{k=1}^\infty \frac{1}{L_{k}^2 L_{k+1}^2} = 5 \sum_{k=1}^\infty (-1)^{k-1} \frac{ F_k F_{k+1}}{L_{k}L_{k+1}^2L_{k+2}},
		\end{gather*}
	and 
		\begin{gather*}
		\sum_{k=1}^\infty \frac{(-1)^k}{F_{k} F_{k+1} F_{2k}F_{2k+2}} = \sum_{k=1}^\infty (-1)^{k} \frac{L_k}{F_{k+1}^2 F_{k+2}},\\
		\sum_{k=1}^\infty \frac{(-1)^k}{L_{k} L_{k+1} L_{2k}L_{2k+2}} = \frac{5}{3} \sum_{k=1}^\infty   (-1)^{k} \frac{F_kF_{k+1}}{L_{k+1} L_{k+2}L_{2k+2}},\\
		\sum_{k=1}^\infty \frac{(-1)^k}{F_{2k}^2 F_{2k+2}^2} = 3\sum_{k=1}^\infty  \frac{(-1)^{k}}{F_{2k+2}F_{2k+4}},\qquad\qquad
		\sum_{k=1}^\infty \frac{(-1)^k}{L_{2k}^2 L_{2k+2}^2} = \sum_{k=1}^\infty (-1)^{k} \frac{F_{2k}}{L_{2k} L_{2k+2}L_{2k+4}}.
		\end{gather*}
\end{corollary}
\begin{proof}
	Apply \eqref{Th3_1} with $m=1$, and \eqref{Th3_2} with $m=2$ and $(r,s,l)=(1,0,0)$ and $(2,0,0)$, considering separately the cases of the Fibonacci and Lucas numbers.
\end{proof}

Building on the ideas of Theorems 1-3, we extend these results to the case of an arbitrary finite number of factors in the denominator. The proof follows along similar lines and is therefore omitted. 
\begin{theorem}\label{thm:n_factors}
	Assume that all series involved converge. Let $r$, $s$, and $l$ are nonnegative integers. 
	
	If $m$ is odd, then
		\begin{gather}\label{Th4_1}
		\sum_{k=1}^\infty 
		\frac{1}{w_{mk+s}w_{m(k+1)+s}\prod\limits_{j=0}^{n-3}w_{r(k+j)+l}}
		= \frac{1}{u_m w_{m+s}}\sum_{k=1}^\infty (-1)^{k-1}
		\frac{\big(w_{rk+l}+w_{r(k+n-2)+l}\big)u_{mk}}
		{w_{m(k+1)+s}\prod\limits_{j=0}^{n-2} w_{r(k+j)+l}};
		\end{gather}
	while if $m$ is even, then
	\begin{gather}\label{Th4_2}
		\sum_{k=1}^\infty 
		\frac{(-1)^k}{w_{mk+s}w_{m(k+1)+s}\prod\limits_{j=0}^{n-3}w_{r(k+j)+l}}
		= \frac{1}{u_m w_{m+s}}\sum_{k=1}^\infty (-1)^{k}
		\frac{\big(w_{rk+l}+w_{r(k+n-2)+l}\big)u_{mk}}
		{w_{m(k+1)+s}\prod\limits_{j=0}^{n-2} w_{r(k+j)+l}}.
		\end{gather}
\end{theorem}

We now present several explicit identities for Fibonacci and Lucas numbers that follow from Theorem \ref{thm:n_factors} 
for cases $n=5$ and $n=6$: 
\begin{corollary}
	\begin{gather*}
		\sum_{k=1}^\infty \frac{1}{F_{k} F_{k+1} F_{k+2}F_{k+3}F_{k+4}} = \frac23\sum_{k=1}^\infty  \frac{(-1)^{k-1}}{F_{k+1} F_{k+3}F_{k+4}},\\
		\sum_{k=1}^\infty \frac{1}{L_{k} L_{k+1} L_{k+2}L_{k+3}L_{k+4}} = \frac27\sum_{k=1}^\infty (-1)^{k-1} \frac{F_k}{L_kL_{k+1} L_{k+3}L_{k+4}},\\
		\sum_{k=1}^\infty \frac{(-1)^k}{F_{2k} F_{2k+2}F_{2k+4}F_{2k+6}F_{2k+8}} = \frac2{21}\sum_{k=1}^\infty (-1)^{k} \frac{L_{2k+3}}{F_{2k+2} F_{2k+4}F_{2k+6}F_{2k+8}},\\
		\sum_{k=1}^\infty \frac{(-1)^k}{L_{2k} L_{2k+2}L_{2k+4}L_{2k+6}L_{2k+8}} 
		= \frac{10}{47}\sum_{k=1}^\infty  (-1)^{k} \frac{F_{2k}F_{2k+3}}{L_{2k}L_{2k+2} L_{2k+4}L_{2k+6}L_{2k+8}},
		\end{gather*}
	and
		\begin{gather*}
		\sum_{k=1}^\infty \frac{1}{F_{k} F_{k+1} F_{k+2}F_{k+3}F_{k+4}F_{k+5}} = \frac35\sum_{k=1}^\infty \frac{(-1)^{k-1}}{F_{k+1} F_{k+3}F_{k+4}F_{k+5}},\\
		\sum_{k=1}^\infty \frac{1}{L_{k} L_{k+1} L_{k+2}L_{k+3}L_{k+4}L_{k+5}} 
		= \frac3{11}\sum_{k=1}^\infty (-1)^{k-1} \frac{F_k}{L_kL_{k+1} L_{k+3}L_{k+4}L_{k+5}},\\
		\sum_{k=1}^\infty \frac{(-1)^k}{F_{2k} F_{2k+2}F_{2k+4}F_{2k+6}F_{2k+8}F_{2k+10}}
		= \frac7{55}\sum_{k=1}^\infty \frac{(-1)^{k}}{F_{2k+2} F_{2k+6}F_{2k+8}F_{2k+10}},\\
		\sum_{k=1}^\infty \frac{(-1)^k}{L_{2k} L_{2k+2}L_{2k+4}L_{2k+6}L_{2k+8} L_{2k+10}}
		= \frac{7}{123}\sum_{k=1}^\infty (-1)^{k} \frac{F_{2k}}{L_{2k}L_{2k+2} L_{2k+6}L_{2k+8}L_{2k+10}}.
		\end{gather*}
\end{corollary}
\begin{proof}
	The series follows by applying \eqref{Th4_1} with $m=1$ and \eqref{Th4_2} with $m=2$ choosing $(r,s,l)=(1,0,0)$ and $(2,0,0)$, respectively, for the Fibonacci and Lucas numbers.
\end{proof}

The identities from Theorem \ref{thm:n_factors} demonstrate that the transformation method extends in a natural way to higher-order reciprocal products, yielding expressions that involve only local combinations of consecutive terms and are often more amenable to further simplification.

\section{Applications to Classical Sequences and Combinatorial Identities}\label{Sect4}

The Abel-type transformations established in the previous section provide a unified framework for deriving a variety of identities involving reciprocal series. We begin by proving a general theorem from which numerous applications to classical sequences and combinatorial identities follow.
\begin{theorem}\label{yjr6p1s}
	Let $(y_k)$ be a sequence of complex numbers and let $Y_n=y_1+y_2+\cdots+ y_n$. Let $m$ and $r$ be non-negative integers. 
	
	If $m$ is odd, then
		\begin{equation}\label{wargel4}
		\sum_{k = 1}^\infty  \frac{y_k}{\prod\limits_{j = 0}^{2m - 1} {w_{k + r + j} } }  = v_m\sum_{k = 1}^\infty  \frac{Y_k
		}{\prod\limits_{j = 0}^{m - 1} {w_{k + r + j} } \prod\limits_{j = m + 1}^{2m} {w_{k + r + j}}};
		\end{equation}
	while if $m$ is even, then
		\begin{equation}\label{clghedx}
		\sum_{k = 1}^\infty  {\frac{{y_k }}{{\prod\limits_{j = 0}^{2m - 1} {w_{k + r + j} } }}}  = u_m\sum_{k = 1}^\infty  \frac{\big( pw_{k + m + r }  + 2w_{k + m + r - 1}  \big)Y_k
		}{{\prod\limits_{j = 0}^{2m} {w_{k + r + j} } }}  .
		\end{equation}
	
	In particular,
		\begin{equation*}
		\sum_{k = 1}^\infty  {\frac{{y_k }}{{w_{k + r} w_{k + r + 1} }}}  = p \sum_{k = 1}^\infty \frac{Y_k}{{w_{k + r} w_{k + r + 2} }} , 
		\end{equation*}
	and
		\begin{equation*}
		\sum_{k = 1}^\infty  {\frac{{y_k }}{{w_{k + r} w_{k + r + 1} w_{k + r + 2} w_{k + r + 3} }}}  = p \sum_{k = 1}^\infty  \frac{{\big( {pw_{k + r + 2}  + 2w_{k + r + 1} } \big)Y_k}}{w_{k + r} w_{k + r + 1} w_{k + r + 2}w_{k + r + 3}w_{k + r + 4}}.
		\end{equation*}
\end{theorem}
\begin{proof} 
	Substituting in \eqref{Lem1_infty}
	$x_k = \big(\prod_{j = 0}^{2m - 1} {w_{k + r + j} }\big)^{-1}$
	and using 
		\begin{equation*}
		x_k  - x_{k + 1}  = (w_{k + r + 2m}  - w_{k + r})\Big(\prod\limits_{j = 0}^{2m - 1} {w_{k + r + j} } \Big)^{-1},
		\end{equation*}
	together with 
	$$w_{m + n} + (- 1)^n w_{m - n} = v_n w_m,$$ and $$w_{m + n} - (- 1)^n w_{m - n} = u_n ( pw_{m} + 2w_{m - 1}),$$
	we get \eqref{wargel4} and \eqref{clghedx}. 
\end{proof}
\begin{corollary}\label{zihreqe}
	Let $(y_n)$ be a sequence of complex numbers. Let $m$ and $r$ be non-negative integers. 
	
	If $m$ is odd, then
	\begin{equation*}
		\sum_{k = 1}^\infty  {\frac{{y_k  - y_{k - 1} }}{{\prod\limits_{j = 0}^{2m - 1} {w_{k + r + j} } }}}  =v_m \sum_{k = 1}^\infty  {\frac{y_k  - y_0 }{\prod\limits_{j = 0}^{m - 1} {w_{k + r + j} } \prod\limits_{j = m + 1}^{2m} w_{k + r + j} }}
		\end{equation*}
	and
		\begin{equation}\label{y+y}
		\sum_{k = 1}^{\infty}  \frac{( - 1)^{k - 1} ( y_k  + y_{k - 1})}{\prod\limits_{j = 0}^{2m - 1} w_{k + r + j}}  = v_m \sum_{k = 1}^{\infty}  \frac{{( - 1)^{k - 1} y_k  + y_0 }}{\prod\limits_{j = 0}^{m - 1} {w_{k + r + j} } \prod\limits_{j = m + 1}^{2m} {w_{k + r + j}}};
		\end{equation}
	while if $m$ is even, then
		\begin{equation*}
		\sum_{k = 1}^\infty  {\frac{{y_k  - y_{k - 1} }}{\prod\limits_{j = 0}^{2m - 1} {w_{k + r + j} } }}  = u_m\sum_{k = 1}^\infty  \frac{{( {w_{k + m + r - 1}  + w_{k + m + r + 1} })( y_k  - y_0 )}}{\prod\limits_{j = 0}^{2m} w_{k + r + j}  }
		\end{equation*}
		and
		\begin{equation*}
		\sum_{k = 1}^\infty  {(-1)^{k-1}\frac{{y_k  + y_{k - 1} }}{\prod\limits_{j = 0}^{2m - 1} {w_{k + r + j} } }}  = u_m \sum_{k = 1}^\infty  {\frac{{ ( {w_{k + m + r - 1}  + w_{k + m + r + 1} } )( {( - 1)^{k - 1}y_k  + y_0 } )}}{{\prod\limits_{j = 0}^{2m} {w_{k + r + j} } }}} .
		\end{equation*}
\end{corollary}
\begin{proof}
	Use $y_k=y^k$ in Theorem \ref{yjr6p1s}.
\end{proof}
\begin{proposition}
	Let $m$, $r$, and $s$ be non-negative integers. 
	
	If $m$ is odd, then
		\begin{equation}
		\sum_{k = 1}^\infty  {\frac{{(p-1)w_{k + s+1} + w_{k+s} }}{{\prod\limits_{j = 0}^{2m - 1} {w_{k + r + j} } }}}  = v_m \sum_{k = 1}^\infty  {\frac{{ {w_{k + s + 2}  - w_{s + 2} } }}{{\prod\limits_{j = 0}^{m - 1} {w_{k + r + j} } \prod\limits_{j = m + 1}^{2m} {w_{k + r + j} } }}} ;
		\end{equation}
	while if $m$ is even, then
		\begin{equation*}
		\sum_{k = 1}^\infty  {\frac{{(p-1)w_{k + s+1}+w_{k+s}}}{{\prod\limits_{j = 0}^{2m - 1} {w_{k + r + j} } }}}  = u_m\sum_{k = 1}^\infty  {\frac{{ \left( {p w_{k + m + r}  + 2 w_{k + m + r - 1} } \right)\left( {w_{k + s + 2}  - w_{s + 2} } \right)}}{\prod\limits_{j = 0}^{2m} w_{k + r + j} }} .
		\end{equation*}
	
	In particular,
		\begin{equation*}
		\sum_{k = 1}^\infty  {\frac{{F_{k + s} }}{{F_{k + r} F_{k + r + 1} }}}  = \sum_{k = 1}^\infty  {\frac{{F_{k + s + 2}  - F_{s + 2} }}{{F_{k + r} F_{k + r + 2} }}}  
		\end{equation*}
	and
		\begin{equation*}
		\sum_{k = 1}^\infty  \frac{{F_{k + s} }}{{F_{k + r} F_{k + r + 1} F_{k + r + 2} F_{k + r + 3} }}  = \sum_{k = 1}^\infty  \frac{ L_{k + r + 2}( {F_{k + s + 2}  - F_{s + 2} } )}{F_{k + r} F_{k + r + 1} F_{k + r + 2}F_{k + r + 3}F_{k + r + 4}} .
		\end{equation*}
\end{proposition}
\begin{proof}
	Use $y_k=w_{k+s+2}$ in Corollary \ref{zihreqe}.
\end{proof}
\begin{proposition}
	If $m$ is an odd integer and $r$ is an integer, then
	\begin{equation*}
		\sum_{k = 1}^\infty  {\frac{{( - 1)^{k - 1} }}{{\prod\limits_{j = 0}^{2m - 1} {w_{k + r + j} } }}}  = v_m \sum_{k = 1}^\infty\frac{{1 }}{{\prod\limits_{j = 0}^{m - 1} {w_{2k  + r + j - 1} } \prod\limits_{j = m + 1}^{2m} {w_{2k  + r + j - 1} } }}.
		\end{equation*}
\end{proposition}
\begin{proof}	Use $y_k=(-1)^k$ in Theorem \ref{yjr6p1s}.
\end{proof}
\begin{proposition}
	If $m$ is an odd integer and $s$ is an integer, then
		\begin{equation*}
		\sum_{k = 1}^\infty  \frac{{w_{k + s}^2 }}{\prod\limits_{j = 0}^{2m - 1} w_{k + r + j}}  =\frac{v_m}{p}  \sum_{k = 1}^\infty\frac{{ {w_{k + s + 1} w_{k + s}  - w_{s + 1} w_s } }}{\prod\limits_{j = 0}^{m - 1} {w_{k + r + j} } \prod\limits_{j = m + 1}^{2m} {w_{k + r + j} }}.
		\end{equation*}
\end{proposition}
\begin{proof}
	Use $y_k=w_{k+s+1}w_{k+s}$ in Corollary \ref{zihreqe}.
\end{proof}
\begin{proposition}
	If $m$ is an odd integer and $s\ne0$, $r$ are integers, then
	\begin{equation*}
		\sum_{k = 1}^\infty  \frac{(-1)^{s2^{k-1}}}{F_{s2^k}\prod\limits_{j = 0}^{2m - 1} F_{k + r + j}}  = \frac{(-1)^sL_m}{F_s}  \sum_{k = 1}^\infty \frac{{ F_{s(2^k-1)}}}{F_{s2^k}\prod\limits_{j = 0}^{m - 1} {F_{k + r + j} } \prod\limits_{j = m + 1}^{2m} {F_{k + r + j} }}.
		\end{equation*}
\end{proposition}
\begin{proof}
	Substitute  $y_k={(-1)^{s2^{k-1}}}/F_{s2^k}$ in Theorem \ref{yjr6p1s} and use the formula 
		\begin{equation*}
		\sum_{j=1}^{n}\frac{(-1)^{s2^{j-1}}}{F_{s2^j}} = \frac{(-1)^sF_{s(2^n-1)}}{F_sF_{s2^n}},
		\end{equation*}
	which ones can find in \cite[Eq. (89)]{Vajda}.
\end{proof}

Let $H_n=\sum\limits_{i=1}^{n}\frac{1}{i}$ be the $n$th harmonic number.
\begin{proposition}
	If $r$ is a non-negative integer and $m$ is an odd integer, then
		\begin{equation*}
		\sum_{k = 1}^\infty  \frac{{H_k }}{{\prod\limits_{j = 0}^{2m - 1} {w_{k + r + j} } }}  = v_m \sum_{k = 1}^\infty  \frac{(k + 1 )H_k  - k}{\prod\limits_{j = 0}^{m - 1} {w_{2k - 1 + r + j}} \prod\limits_{j = m + 1}^{2m} {w_{2k - 1 + r + j} } } .
		\end{equation*}
\end{proposition}
\begin{proof}
	Use $y_k=H_k$ in Theorem \ref{yjr6p1s} together with
	$\sum_{j=1}^kH_j=(k+1)H_k-k$.
\end{proof}

In the next proposition, we use the Catalan numbers, defined by 
$C_j=\frac{1}{j+1}\binom{2j}{j}$, for nonnegative $j$.
\begin{proposition}
	If $q$ and $s$ are complex numbers and $m$ is an odd integer, then
		\begin{equation}\label{Prop6_1}
		\sum_{k = 1}^\infty  {\frac{{\binom{{k + q}}s}}{{\prod\limits_{j = 0}^{2m - 1} {w_{k + r + j} } }}}  = v_m \sum_{k = 1}^\infty  {\frac{\binom{k + q + 1}{s + 1} - \binom{q + 1}{s + 1}}{{\prod\limits_{j = 0}^{m - 1} {w_{k + r + j} } \prod\limits_{j = m + 1}^{2m} {w_{k + r + j} } }}} 
		\end{equation}
	and
		\begin{equation}\label{Prop6_2}
		\sum_{k = 1}^\infty  {\frac{{( - 1)^{k - 1} \binom{q}{k + s}}}{\prod\limits_{j = 0}^{2m - 1} w_{k + r + j}}}  = v_m \sum_{k = 1}^\infty  \frac{{ {( - 1)^{k - 1} \binom{{q - 1}}{{k + s}} + \binom{{q - 1}}{s}}}}{{\prod\limits_{j = 0}^{m - 1} {w_{k + r + j} } \prod\limits_{j = m + 1}^{2m} {w_{k + r + j} } }}  .
		\end{equation}
	
	In particular, if $m$ is an odd integer, then
		\begin{gather*}
		\sum_{k = 1}^\infty  {\frac{{2^{2k} }}{{\binom{{2k}}{k}\prod\limits_{j = 0}^{2m - 1} {w_{k + r + j} } }}}  = \frac{2v_m}{3} \sum_{k = 1}^\infty  {\frac{{2^{2k}  - C_k }}{{C_k \prod\limits_{j = 0}^{m - 1} {w_{k + r + j} } \prod\limits_{j = m + 1}^{2m} {w_{k + r + j} } }}},\\ 
		\sum_{k = 1}^\infty  {\frac{{\binom{2k}{k}}}{{2^{2k} \prod\limits_{j = 0}^{2m - 1} {w_{k + r + j} } }}}  = v_m \sum_{k = 1}^\infty  {\frac{{(k+1)( 2k + 1 )2^{-2k}C_k - 1 }}{{ \prod\limits_{j = 0}^{m - 1} {w_{k + r + j} } \prod\limits_{j = m + 1}^{2m} {w_{k + r + j} } }}} .
		\end{gather*}
\end{proposition}
\begin{proof}
	Use Pascal's formula
	$\binom{k+q}s=\binom{k+q+1}{s+1}-\binom{k+q}{s+1}$
	in Corollary \ref{zihreqe}.
\end{proof}
\begin{corollary} Let $q$ and $s$ be complex numbers and m be an odd integer. Then
\begin{align}
		&\sum_{k = 1}^\infty  \frac{\binom{k+q}{s}(H_{k+q}-H_{k+q-s})}{\prod\limits_{j = 0}^{2m - 1} w_{k + r + j}} \notag \\
		&\qquad\qquad\qquad  = v_m \sum_{k = 1}^\infty 
		\frac{\binom{k+q+1}{s+1}(H_{k+q+1}-H_{k+q-s})- \binom{q+1}{s+1}(H_{q+1}-H_{q-s})}{\prod\limits_{j = 0}^{m - 1} {w_{k + r + j} }\prod\limits_{j = m + 1}^{2m} {w_{k + r + j} } }\label{Cor13_1_diff_q}
		\end{align}
	and
		\begin{align}
		&\sum_{k = 1}^\infty (-1)^{k-1}
		\frac{
			\binom{q}{k+s}
			(H_q-H_{q-k-s})}
		{\prod\limits_{j = 0}^{2m - 1} w_{k+r+j}}\notag \\
		&\qquad\qquad\qquad=v_m
		\sum_{k = 1}^\infty
		\frac{(-1)^{k-1}
			\binom{q-1}{k+s}
			(H_{q-1}-H_{q-k-s-1}) +
			\binom{q-1}{s}
			(H_{q-1}-H_{q-s-1})}
		{\prod\limits_{j = 0}^{m - 1} w_{k+r+j}
			\prod\limits_{j = m + 1}^{2m} w_{k+r+j}
		}.
		\label{Cor13_2_diff_q}
		\end{align}
	
	In particular, by setting $s=1$ in  \eqref{Cor13_1_diff_q}, for odd integer $m$, we have 
			\begin{equation*}
		\sum_{k = 1}^\infty  \frac{1}{\prod\limits_{j = 0}^{2m - 1} w_{k + r + j}} = v_m \sum_{k = 1}^\infty 
		\frac{k}{\prod\limits_{j = 0}^{m - 1} {w_{k + r + j} }\prod\limits_{j = m + 1}^{2m} {w_{k + r + j} }}.
		\end{equation*}
		\begin{proof}
		Differentiate  \eqref{Prop6_1} and \eqref{Prop6_2} with respect to $q$, using  
		$\frac{\partial}{\partial m}\binom{m}{n}=\binom{m}{n}(H_m-H_{m-n}).$
	\end{proof}
\end{corollary}
\begin{corollary} Let $q$ and $s$ be complex numbers and m be an odd integer. Then
	\begin{align}\label{Cor9_1}
		&\sum_{k=1}^{\infty} \frac{\binom{k+q}{s} \big( H_{k+q-s} - H_{s} \big)}{\prod\limits_{j=0}^{2m-1} w_{k+r+j}}\notag\\
		&\qquad\qquad\qquad= v_m \sum_{k=1}^{\infty} \frac{ \binom{k+q+1}{s+1} \big( H_{k+q-s} - H_{s+1} \big) - \binom{q+1}{s+1} \left( H_{q-s} - H_{s+1} \right) }{\prod\limits_{j=0}^{m-1} w_{k+r+j} \prod\limits_{j=m+1}^{2m} w_{k+r+j}}
		\end{align}
	and 
\begin{align}\label{Cor9_2}
		&\sum_{k=1}^{\infty} (-1)^{k-1}\frac{ \binom{q}{k+s} \big( H_{k+s} - H_{q-k-s} \big)}{\prod\limits_{j=0}^{2m-1} w_{k+r+j}}\notag \\
		&\qquad\qquad\qquad= v_m \sum_{k=1}^{\infty} \frac{ (-1)^{k-1}\binom{q-1}{k+s} \big( H_{k+s} - H_{q-k-s-1} \big) + \binom{q-1}{s} ( H_{s} - H_{q-s-1} )}{\prod\limits_{j=0}^{m-1} w_{k+r+j} \prod\limits_{j=m+1}^{2m} w_{k+r+j}} .
		\end{align}
	\end{corollary}
\begin{proof} Differentiate  \eqref{Prop6_1} and \eqref{Prop6_2} with respect to $s$, using  
	$\frac{\partial}{\partial n}\binom{m}{n}=\binom{m}{n}(H_{m-n}-H_{n})$.
\end{proof}

Note that identities \eqref{Cor13_1_diff_q}--\eqref{Cor9_2}  can be rewritten in terms of the odd harmonic numbers. Indeed, by substituting $q\mapsto q-\frac{1}{2}$
into   \eqref{Cor13_1_diff_q} and,  \eqref{Cor13_2_diff_q} and $s\mapsto s-\frac{1}{2}$
into   \eqref{Cor9_1} and \eqref{Cor9_2}, and using the identity $H_{s-1/2}=2 O_s+2\ln 2$, we obtain, respectively,
	\begin{align*}
	&\sum_{k = 1}^\infty  \frac{\binom{k+q-\frac12}{s}(O_{k+q}-O_{k+q-s})}{\prod\limits_{j = 0}^{2m - 1} w_{k + r + j}}\\
	&\qquad\qquad = v_m \sum_{k = 1}^\infty 
	\frac{\binom{k+q+\frac12}{s+1}(O_{k+q+1}-O_{k+q-s})- \binom{q+\frac12}{s+1}(O_{q+1}-O_{q-s})}{\prod\limits_{j = 0}^{m - 1} {w_{k + r + j} }\prod\limits_{j = m + 1}^{2m} {w_{k + r + j} } },
	\end{align*}
	\begin{align*}	
	&\sum_{k = 1}^\infty
	\frac{(-1)^{k-1}
		\binom{q-\frac12}{k+s}
		(O_q-O_{q-k-s})}
	{\prod\limits_{j = 0}^{2m - 1} w_{k+r+j}}
	\\
	&\qquad\qquad=v_m
	\sum_{k = 1}^\infty
	\frac{(-1)^{k-1}
		\binom{q-\frac32}{k+s}
		(O_{q-1}-O_{q-k-s-1})+
		\binom{q-\frac32}{s}
		(O_{q-1}-O_{q-s-1})}
	{\prod\limits_{j = 0}^{m - 1} w_{k+r+j}
		\prod\limits_{j = m + 1}^{2m} w_{k+r+j}
	},\\
	&\sum_{k=1}^{\infty} \frac{\binom{k+q}{s-\frac12} \big( O_{k+q-s+1} - O_{s} \big)}{\prod\limits_{j=0}^{2m-1} w_{k+r+j}}\notag\\
	&\qquad\qquad\qquad= v_m \sum_{k=1}^{\infty} \frac{ \binom{k+q+1}{s+\frac12} \big( O_{k+q-s+1} - O_{s+1} \big) - \binom{q+1}{s+\frac12} \left( O_{q-s+1} - O_{s+1} \right) }{\prod\limits_{j=0}^{m-1} w_{k+r+j} \prod\limits_{j=m+1}^{2m} w_{k+r+j}}, 
	\end{align*}
and 
\begin{align*}
	&\sum_{k=1}^{\infty} (-1)^{k-1}\frac{ \binom{q}{k+s-\frac12} \big( O_{q-k-s+1} - O_{k+s} \big)}{\prod\limits_{j=0}^{2m-1} w_{k+r+j}}\notag \\
	&\qquad\qquad\qquad= v_m \sum_{k=1}^{\infty} \frac{ (-1)^{k-1}\binom{q-1}{k+s-\frac12} \big( O_{q-k-s} - O_{k+s}\big) + \binom{q-1}{s-\frac12} ( O_{q-s}  - O_{s})}{\prod\limits_{j=0}^{m-1} w_{k+r+j} \prod\limits_{j=m+1}^{2m} w_{k+r+j}} .
	\end{align*}

Below, $[\cdot]$ and $\{\cdot\}$ denote the unsigned Stirling numbers of the first kind and the Stirling numbers of the second kind, respectively. They are defined by
\[
x^{\overline{n}}
=\sum_{k=0}^{n}\left[{n\atop k}\right]x^k,
\qquad
x^n
=\sum_{k=0}^{n}\left\{{n\atop k}\right\}x^{\underline{k}},
\]
where
$x^{\overline{n}}=x(x+1)\cdots(x+n-1)$ and $x^{\underline{n}}=x(x-1)\cdots(x-n+1)$ denote the rising and the falling  factorials, respectively; see \cite{Graham}. 
\begin{proposition}
	If $m$ is an odd integer, $n$ is a positive integer, and $r$ is an integer, then
	\begin{equation*}
		\sum_{k=1}^{\infty}
		\frac{(k+n)\left[{k+n\atop k}\right]}
		{\prod\limits_{j=0}^{2m-1}w_{k+r+j}}
		=
		v_m
		\sum_{k=1}^{\infty}
		\frac{\left[{k+n+1\atop k}\right]}
		{\prod\limits_{j=0}^{m-1}w_{k+r+j}
			\prod\limits_{j=m+1}^{2m}w_{k+r+j}}
		\end{equation*}
and
\begin{equation*}
		\sum_{k=1}^{\infty}
		\frac{k\left\{{k+n\atop k}\right\}}
		{\prod\limits_{j=0}^{2m-1}w_{k+r+j}}
		=
		v_m
		\sum_{k=1}^{\infty}
		\frac{\left\{{k+n+1\atop k}\right\}}
		{\prod\limits_{j=0}^{m-1}w_{k+r+j}
			\prod\limits_{j=m+1}^{2m}w_{k+r+j}}.
		\end{equation*}
\end{proposition}
\begin{proof}
	Use the identities \cite[p.~251]{Graham}
	\[
	\sum_{j=1}^{k}(j+n)
	\left[{j+n\atop j}\right]
	=
	\left[{k+n+1\atop k}\right],
	\qquad
	\sum_{j=1}^{k}
	j\left\{{j+n\atop j}\right\}
	=
	\left\{{k+n+1\atop k}\right\},
	\]
	together with Theorem \ref{yjr6p1s}.
\end{proof}

From Corollary \ref{zihreqe}, we can also derive numerous formulas related to polynomials. We present only those related to the Chebyshev polynomials of the first and second kinds, $T_n(x)$ and $U_n
(x)$, respectively, which are defined recursively by
	\begin{equation*}
	U_0=1,\,\, U_1(x)=2x,\quad U_n(x)=2xU_{n-1}(x)-U_{n-2}(x),\qquad n\geq2,
	\end{equation*}
and
	\begin{equation*}
	T_0=1,\,\, T_1(x)=x,\quad T_n(x)=2xT_{n-1}(x)-T_{n-2}(x),\qquad n\geq2.
	\end{equation*}
\begin{proposition}
	If $m$ is an odd integer, $n$ is a positive integer and $r$ is an integer, then
		\begin{equation*}
		2x\sum_{k = 1}^\infty (-1)^{k-1}  \frac{U^2_{k+s}(x)}{\prod\limits_{j = 0}^{2m - 1} w_{k + r + j}}  =  v_m \sum_{k = 1}^\infty \frac{(-1)^{k-1}U_{k+s+1}(x)U_{k+s}(x) + U_{s+1}(x)U_{s}(x)}{\prod\limits_{j = 0}^{m - 1} w_{k + r + j}  \prod\limits_{j = m + 1}^{2m} w_{k + r + j}}
		\end{equation*}
and
	\begin{equation*}
		2x\sum_{k = 1}^\infty (-1)^{k-1}  \frac{T^2_{k+s}(x)}{\prod\limits_{j = 0}^{2m - 1} w_{k + r + j}}  =  v_m \sum_{k = 1}^\infty \frac{(-1)^{k-1}T_{k+s+1}(x)T_{k+s}(x) + T_{s+1}(x)T_{s}(x)}{\prod\limits_{j = 0}^{m - 1} w_{k + r + j}  \prod\limits_{j = m + 1}^{2m} w_{k + r + j}},
		\end{equation*}
where 
	$|x|\le 2^{-m-1}\Big(\big(p+\sqrt{p^2+4}\big)^m+\big(-p+\sqrt{p^2+4}\big)^{m}\Big)$.
\end{proposition}
\begin{proof}
	Use $y_k=U_{k+s+1}U_{k+s}$ or $y_k=T_{k+s+1}T_{k+s}$ in \eqref{y+y}.
\end{proof}

\section{Additional Applications}

The Abel-type transformation established in Lemma \ref{Lem1_infty} possesses a remarkable symmetry, which immediately yields the following identity.
\begin{proposition}\label{new}
	Under the assumptions of Lemma \ref{Lem1_infty},
	\begin{equation}\label{einzt6m}
	\sum_{k=1}^{\infty}
	\left(x_k-x_{k+1}\right)
	\sum_{j=1}^{k}y_j
	=
	\sum_{k=1}^{\infty}
	\left(y_k-y_{k+1}\right)
	\sum_{j=1}^{k}x_j.
	\end{equation}
\end{proposition}
\begin{proof}
	The identity follows immediately by interchanging the sequences
	\(x_j\) and \(y_j\) in \eqref{Lem1_infty}. Moreover,
	\[
	\lim_{n\to\infty}
	x_{n+1}\sum_{j=1}^{n}y_j=0,
	\qquad
	\lim_{n\to\infty}
	y_{n+1}\sum_{j=1}^{n}x_j=0,
	\]
	so the assumptions of Lemma \ref{Lem1_infty} remain satisfied.
\end{proof}

The following examples illustrate several applications of Proposition \ref{new}.
\begin{example} Applying Proposition \ref{new} with
	\[
	(x_j,y_j)
	=
	\left(\frac1{w_j},\frac1{u_j}\right),
	\quad (x_j,y_j)=
	\left(\frac1{w_jw_{j+1}},\frac1{u_ju_{j+1}}\right),
	\quad (x_j,y_j)=
	\left(\frac{(-1)^j}{w_jw_{j+1}},\frac1{u_ju_{j+1}}\right),
	\]
	we obtain, respectively,
	\begin{gather*}
	\sum_{k=1}^{\infty}
	\frac{pw_{k+1}+2w_k}
	{w_kw_{k+1}w_{k+2}}
	\sum_{j=1}^{k}
	\frac1{u_ju_{j+1}}
	=p
	\sum_{k=1}^{\infty}
	\frac1{u_ku_{k+2}}
	\sum_{j=1}^{k}
	\frac{(-1)^j}{w_jw_{j+1}},\notag\\
	\sum_{k=1}^{\infty}
	\frac1{w_kw_{k+2}}
	\sum_{j=1}^{k}
	\frac1{u_ju_{j+1}}
	=\sum_{k=1}^{\infty}
	\frac1{u_ku_{k+2}}
	\sum_{j=1}^{k}
	\frac1{w_jw_{j+1}},\notag\\
	\sum_{k=1}^{\infty}
	(-1)^k
	\frac{pw_{k+1}+2w_k}
	{w_kw_{k+1}w_{k+2}}
	\sum_{j=1}^{k}
	\frac1{u_ju_{j+1}}
	=p
	\sum_{k=1}^{\infty}
	\frac1{u_ku_{k+2}}
	\sum_{j=1}^{k}
	\frac{(-1)^j}{w_jw_{j+1}}.
	\end{gather*}
	
	Using the telescoping evaluation
	\begin{equation*}
	\sum_{j=1}^{k}
	\frac{(-1)^j}{w_jw_{j+1}}
	=
	\frac1{b^2-abp-a^2}
	\left(
	\frac{a}{b}
	-
	\frac{w_k}{w_{k+1}}
	\right),
	\end{equation*}
	the last identity simplifies to
	\begin{equation*}
	\sum_{k=1}^{\infty}
	(-1)^k
	\frac{pw_{k+1}+2w_k}
	{w_kw_{k+1}w_{k+2}}
	\sum_{j=1}^{k}
	\frac1{u_ju_{j+1}}
	=
	\frac{p}{b^2-abp-a^2}
	\sum_{k=1}^{\infty}
	\frac1{u_ku_{k+2}}
	\left(
	\frac{a}{b}
	-
	\frac{w_k}{w_{k+1}}
	\right).
	\end{equation*}
\end{example}

Specializing to the Fibonacci and Lucas numbers, we obtain the following identities.
\begin{corollary}
	\begin{gather*}
	\sum_{k=1}^{\infty}
	\frac1{F_kF_{k+2}}
	\sum_{j=1}^{k}
	\frac1{L_jL_{j+1}}
	=
	\sum_{k=1}^{\infty}
	\frac1{L_kL_{k+2}}
	\sum_{j=1}^{k}
	\frac1{F_jF_{j+1}},\\
	\sum_{k=1}^{\infty}
	\frac{(-1)^k}{F_{k+1}F_{k+2}}
	=
	5
	\sum_{k=1}^{\infty}
	(-1)^k
	\frac{F_k}{L_kL_{k+1}L_{k+2}}
	\sum_{j=1}^{k}
	\frac1{F_jF_{j+1}}.
	\end{gather*}
\end{corollary}

\section{Conclusion}

We have developed a unified method for transforming infinite series involving products of terms of second-order linear recurrences in the denominator. The method is based on a discrete Abel-type summation formula, which converts such series into expressions with a partial telescoping or alternating structure.

Using this approach, we obtained general transformation formulas for reciprocal series with three factors and extended them to products of four or more terms. These results provide a systematic framework that unifies and extends many known identities for Fibonacci and Lucas numbers.

In addition, the method yields explicit identities involving several classical combinatorial sequences, including Catalan numbers, harmonic numbers, and Stirling numbers of both kinds. This shows that the approach is flexible and connects reciprocal series with different areas of combinatorics.

A notable feature of the method is that it naturally distinguishes between even and odd values of the parameters, leading to different types of representations. This reflects an inherent structural property of reciprocal series associated with second-order recurrences.

Possible directions for future work include extensions to more general recurrence relations, $q$-analogues, and further connections with generating functions and special functions.

\end{document}